\documentclass[12pt]{amsart}
\usepackage{amsfonts,latexsym,amsthm,amssymb,graphicx}
\usepackage{wrapfig}
\usepackage[all]{xy}
\DeclareFontFamily{OT1}{rsfs}{}
\DeclareFontShape{OT1}{rsfs}{n}{it}{<-> rsfs10}{}
\DeclareMathAlphabet{\mathscr}{OT1}{rsfs}{n}{it}

\CompileMatrices

\newtheorem{theorem}{Theorem}[section]
\newtheorem{lemma}[theorem]{Lemma}

\newtheorem{claim}[theorem]{Claim}

{\theoremstyle{definition} }
{\theoremstyle{remark} \newtheorem{remark}[theorem]{Remark}
\newtheorem{example}[theorem]{Example}}

\newcommand{\Cbb}{{\mathbb{C}}}

\newcommand{\Pbb}{{\mathbb{P}}}

\newcommand{\cA}{{\mathscr A}}

\newcommand{\cE}{{\mathscr E}}
\newcommand{\cF}{{\mathscr F}}

\newcommand{\cO}{{\mathscr O}}

\newcommand{\Til}[1]{{\widetilde{#1}}}
\newcommand{\one}{1\hskip-3.5pt1}
\newcommand{\csm}{{c_{\text{SM}}}}

\DeclareMathOperator{\im}{im}
\DeclareMathOperator{\codim}{codim}

\DeclareMathOperator{\Tor}{Tor}

\newcommand{\qede}{\hfill$\lrcorner$}

\newcommand{\mustata}{{Musta\c{t}\v{a}}}

\DeclareMathOperator{\Der}{Der}

\title{Chern classes of free hypersurface arrangements}
\author{Paolo Aluffi}
\address{
Mathematics Department, 
Florida State University,
Tallahassee FL 32306, U.S.A.
}
\email{aluffi@math.fsu.edu}

\begin{document}

\begin{abstract}
The Chern class of the sheaf of logarithmic derivations along a simple normal 
crossing divisor equals the Chern-Schwartz-MacPherson class of the
complement of the divisor. We extend this equality to more general
divisors, which are locally analytically isomorphic to free hyperplane 
arrangements.
\end{abstract}

\maketitle


\section{Introduction}\label{intro}

For us, an {\em arrangement\/} in a nonsingular variety $V$ is a reduced divisor~$D$
consisting of a union of nonsingular hypersurfaces, such that at each point $D$ is 
locally analytically isomorphic to a 
hyperplane arrangement. We say that
the arrangement is {\em free\/} if all these local models may be chosen to be free
hyperplane arrangements. It follows that $D$ is itself a {\em free divisor\/} on $V$:
the sheaf of logarithmic differentials $\Omega_V^1(\log D)$ along $D$ is locally
free. Equivalently, its dual sheaf of {\em logarithmic derivations,\/} 
$\Der_V(D):=\Omega_V^1(\log D)^\vee$, is locally free.
Free hyperplane arrangements in~$\Pbb^n$ and divisors with simple
normal crossings in a nonsingular variety give examples of free hypersurface 
arrangements.

In this note we extend to free hypersurface arrangements a result that is known
to hold for these examples.

\begin{theorem}\label{main}
Let $V$ be a nonsingular compact complex variety, and let $D\subseteq V$ be
a free hypersurface arrangement. Then
\[
c(\Der_V(D))\cap [V] = \csm(\one_{V\smallsetminus D})\quad.
\]
\end{theorem}

Here, $\csm(\one_{V\smallsetminus D})$ is the {\em Chern-Schwartz-MacPherson\/}
class of the constructible function $\one_{V\smallsetminus D}$, in the sense of
\cite{MR0361141}, see also~\cite{85k:14004}, Example~19.1.7.

For simple normal crossing divisors, the equality of Theorem~\ref{main} was verified in
\cite{MR1893006} (Proposition~15.3) and \cite{MR2001d:14008}. For free projective
hyperplane arrangements, it is Theorem~4.1 in~\cite{hyparr}, where it is obtained
as a simple corollary of a result of \mustata{} and Schenck (\cite{MR1843320}). 
Theorem~\ref{main}
will be obtained here by considering the blow-ups giving an embedded resolution of 
$D$. Each blow-up will be analyzed by using MacPherson's {\em graph 
construction,\/} showing (Claim~\ref{mainclaim}) that the Chern class of the
corresponding sheaf of logarithmic derivations is preserved by push-forward.
The theorem will then follow from the corresponding behavior of the 
Chern-Schwartz-MacPherson class and from the case of normal crossing
divisors.

In particular, this will give an independent proof (and a substantial generalization)
of the case of free hyperplane arrangements treated in~\cite{hyparr}.
\medskip

The term `hypersurface arrangement' is often used in the literature to simply
mean a union of hypersurfaces (nonsingular or otherwise). This is a substantially 
more general notion than the one used in this note. The statement of 
Theorem~\ref{main} is {\em not\/} true in this generality, even for free divisors.
For example, if $V$ is a surface (so that every reduced divisor is free in $V$),
a condition of local homogeneity is necessary for this result to hold, as 
observed by Xia Liao (cf.~\cite{liao}).  

The paper is organized as follows: in \S\ref{setup} we recall the basic
definitions and reduce the main theorem to showing that Chern classes of
sheaves of logarithmic derivations are preserved through certain types of
blow-ups. This is proven in \S\ref{proof}, using the graph
construction. In \S\ref{remandex} we offer a simple example, and show 
that the theorem is equivalent to a projection formula for Chern classes
of certain coherent sheaves.

A word on the hypotheses: the freeness of the divisor is used crucially in the
application of the graph construction; its local analytic structure is less
essential, but convenient in some coordinate arguments. It is conceivable
that the proof given here may be generalized to divisors satisfying a less
restrictive local homogeneity requirement.

The result in this note generalizes Theorem~4.1 in~\cite{hyparr}. I presented the
results of~\cite{hyparr} in my talk at the Hefei conference on Singularity Theory, 
and I take this opportunity to thank Xiuxiong Chen and Laurentiu Maxim
for the invitation
to speak at the conference and for organizing a very successful and thoroughly
enjoyable meeting.


\section{Set-up}\label{setup}

\subsection{}\label{csmsum}
We work over an algebraically closed field of characteristic~$0$; the reader is welcome
to assume the ground field is $\Cbb$. (Characteristic~$0$ is required in the theory of 
Chern-Schwartz-MacPherson classes. See~\cite{MR1063344} or~\cite{MR2282409} 
for a discussion of the theory over algebraically closed fields of characteristic~$0$.)

Chern-Schwartz-MacPherson classes are classes in the Chow group of a variety $V$
defined for constructible functions on $V$, and are characterized by the normalization requirement that $\csm(\one_V)\in A_*V$ equals $c(TV)\cap [V]$ if $X$ is nonsingular
and the covariance property 
\[
\alpha_* \csm(\varphi)=\csm(\alpha_* \varphi)
\]
for all proper morphisms $\alpha: V \to V'$ and all constructible functions $\varphi$
on $V$.
Over~$\Cbb$, the push-forward of a constructible function is defined by taking
weighted Euler characteristics of fibers: for a subvariety $Z\subseteq V$,
$\alpha_*(\one_Z)(p)=\chi (Z\cap \alpha^{-1}(p))$. 
Thus, $\csm$ determines a natural transformation from the functor of 
constructible functions to the Chow functor. The existence of this natural 
transformation was conjectured by Deligne and Grothendieck, and proved by
MacPherson (\cite{MR0361141}).
The reader may consult Example~19.1.7 in~\cite{85k:14004} for an efficient
summary of MacPherson's definition. An alternative construction is presented
in~\cite{MR2282409}. Interest in these classes has resurged in the past few
years; comparison
with other classes for singular varieties gives an intersection theoretic invariant of 
singularities generalizing directly the Milnor number. A recent survey may be 
found in~\cite{MR2325145}.

\subsection{}
The covariance property of Chern-Schwartz-MacPherson classes has the
following immediate consequence. Let $V$ be a variety, and let $X\subseteq V$
be a subscheme. Let $\rho: \Til V\to V$ be a proper map, and let $X'\subseteq
\Til V$ be any subscheme such that $\rho$ restricts to an isomorphism
$\Til V\smallsetminus X' \to V\smallsetminus X$. Then
\[
\rho_* \csm(\one_{\Til V\smallsetminus X'}) = \csm(\one_{V\smallsetminus X})
\quad.
\]
Indeed, $\rho_*(\one_{\Til V\smallsetminus X'})=\one_{V\smallsetminus X}$.

In particular:

\begin{lemma}\label{csmred}
Let $V$ be a nonsingular complete variety, and let $D\subseteq V$ be a
subscheme. Let $\rho: \Til V \to V$ be a proper morphism such that 
$\Til V$ is nonsingular, and the support $D'$ of $\rho^{-1}(D)$ is a divisor 
with normal crossings and nonsingular components. Then
\[
\csm(\one_{V\smallsetminus D})=\rho_*(c(\Der_{\Til V}(\log D'))
\cap[\Til V])
\quad.
\]
\end{lemma}

\begin{proof}
As recalled in \S\ref{intro}, since $D'$ is a simple normal crossing divisor 
in $\Til V$, then
\[
\csm(\one_{\Til V\smallsetminus D'}) = c(\Omega^1_{\Til V}(\log D'))^\vee)\cap [\Til V]
= c(\Der_{\Til V}(D'))\cap [\Til V]
\]
(cf.~for example \cite{MR2001d:14008}, Theorem~1). The formula follows then 
immediately from covariance.
\end{proof}

\subsection{}
Now let $V$ be a complete nonsingular variety, and let $D$ be a 
hypersurface
arrangement, as in \S\ref{intro}. In particular: at every $p\in D$, there is a choice of
analytic coordinates $x_1,\dots, x_n$ such that the ideal of $D$ in the
completion $k[[x_1,\dots,x_n]]$ is generated by a product of linear polynomials 
$\sum \lambda_i x_i$, defining a central hyperplane arrangement $\cA_p$.

\begin{lemma}
The divisor $D$ is free on $V$ if and only if each $\cA_p$ is a free central
hyperplane arrangement.
\end{lemma}

\begin{proof}
Recall that a divisor in a nonsingular variety $V$ is free if and only if its singularity 
subscheme is empty or Cohen-Macaulay of codimension~$2$ in $V$ at each 
$p\in D$. It follows that a central hyperplane arrangement is free if and only if its 
singularity subscheme is empty or Cohen-Macaulay of codimension~$2$ at the 
origin. (Cf.~\cite{MR586451}, Proposition~2.4.)

The statement then follows from the fact that a local ring is Cohen-Macaulay if 
and only if its completion is (\cite{MR1251956}, Corollary~2.1.8).
\end{proof}

Under the hypotheses of Theorem~\ref{main}, 
$\Der_V(D)$ is locally
free. With $\rho: \Til V\to V$ as in the statement of Lemma~\ref{csmred}, 
$\Der_{\Til V}(D')$ is also locally free, as $D'$ is a divisor with simple
normal crossings. Lemma~\ref{csmred} reduces Theorem~\ref{main} to proving 
that if $D$ is a free hypersurface arrangement in $V$, and $\rho: \Til V\to V$ is as
in the statement of Lemma~\ref{csmred}, then
\begin{equation*}
\tag{$\dagger$}
\rho_*(c(\Der_{\Til V}(D')) \cap [\Til V])
=c(\Der_V(D)) \cap [V]\quad.
\end{equation*}

\subsection{}
Next, we observe that an embedded resolution $\rho$ of a hypersurface 
arrangement~$D$
may be obtained by blowing up at the intersections of the components of the
arrangement, in order of increasing dimension, and that these intersections
are all nonsingular. In order to verify ($\dagger$), it suffices to verify that 
the stated equality holds for each of these blow-ups. More precisely:
Given a hypersurface arrangement~$D$ in a nonsingular variety $V$,
let $Z$ be a component of lowest dimension among the intersections
of components of $D$; let $\pi: \hat V \to V$ be the blow-up of $V$ along $Z$;
let $E$ be the exceptional divisor of this blow-up; and let $D'$ be the
divisor in $\hat V$ consisting of $E$ and the proper transforms of the 
components of $D$.

\begin{lemma}\label{freeblow}
With notation as above, if $D$ is a free hypersurface arrangement,
then so is $D'$.
\end{lemma}

\begin{proof}
We can work analytically at a point $p\in V$, so we may assume that $D$ is 
given by a product of linear forms cutting out the center $Z$ at $p$.
We may in fact assume that there are analytic coordinates $x_1,\dots,x_n$
at $p$ so that $Z$ is given by $x_1=\cdots=x_r=0$, and the generator of
the ideal of $D$ is a homogeneous polynomial $F(x_1,\dots,x_r)
=\prod L_i(\underline x)$, with $L_i$ linear.

Let $q\in \hat V$ be a point over $p$. We may choose analytic coordinates 
$(\hat x_1,\hat x_2,\dots, \hat x_n)$ at $q$ so that $\hat x_1=0$ is the 
exceptional divisor, and the blow-up map is given by
\[
(x_1,\dots,x_n) = (\hat x_1,\hat x_1 \hat x_2,\dots, \hat x_1\hat x_r,
\hat x_{r+1},\dots,\hat x_n)\quad.
\]
The ideal for $D'$ at $q$ is then generated by
\[
\hat x_1\, F(1,\hat x_2,\dots, \hat x_r)\quad;
\]
omitting the factors in $F(1,\hat x_2,\dots, \hat x_r)$ that do not vanish at
$q$, we write the generator for $D'$ at $q$ as
\[
\hat x_1\, G(\hat x_2,\dots, \hat x_r)\quad,
\]
where $G$ is a product of linear forms.
In particular, $D'$ is a hypersurface arrangement in $\hat V$. We have to verify
that it is free.

First we note that the divisor defined by $G(\hat x_2,\dots, \hat x_r)$ is
free at $q$: indeed, the hyperplane arrangement defined by $F(x_1,\dots, x_n)$
is free by assumption, and $G(x_2,\dots,x_r)$ generates the ideal of this
arrangement at points $(t,0,\dots,0)$ with $t\ne 0$. By Saito's criterion
(\cite{MR1217488}, Theorem~4.19), $G(\hat x_2,\dots,\hat x_r)$ is the 
determinant of a set of $n-1$ logarithmic derivations $\theta_2,\dots, 
\theta_n$ at $q$.
Since $\theta_2(\hat x_1)=\cdots=\theta_n(\hat x_1)=0$,
these derivations are logarithmic with respect to $\hat x_1\, G(\hat x_2,\dots,
\hat x_r)$.

On the other hand the Euler derivation $\theta_1=\hat x_1\,\partial/\partial {\hat x_1} 
+ \hat x_2\, \partial/\partial{\hat x_2}+\cdots + \hat x_n\, \partial/\partial {\hat x_n}$ 
is logarithmic with respect to $\hat x_1 G$ as this is homogeneous
(cf.~\cite{MR1217488}, Definition~4.7),
and $\det(\theta_1,\dots,\theta_n)$ is a unit multiple of $\hat x_1\, G(\hat x_2,
\dots, \hat x_r)$. This shows that $D'$ is free, again by Saito's criterion.
\end{proof}

\subsection{}
By Lemma~\ref{freeblow}, $\Der_{\hat V}(D')$ is locally free
if $\Der_V(D)$ is, and we may consider its ordinary Chern classes. 
We have reduced the proof of Theorem~\ref{main} to the following statement.

\begin{claim}\label{mainclaim}
Let $D$ be a free hypersurface arrangement on a nonsingular variety $V$;
let $\pi: \hat V\to V$ be the blow-up of $V$ along a component of lowest
dimension of the intersection of components of $D$, and let $D'
=(\pi^{-1}(D))_{\text{red}}$, as above. Then
\[
\pi_*(c(\Der_{\hat V}(D')) \cap [\hat V])=c(\Der_V(D)) \cap [V]\quad.
\]
\end{claim}

The next section is devoted to the proof of this claim, and this will complete 
the proof of Theorem~\ref{main}.


\section{Proof of Theorem~\ref{main}}\label{proof}

\subsection{}\label{graphsec}
We will prove Claim~\ref{mainclaim} as an application of MacPherson's {\em graph 
construction.\/} (We refer the reader to Example~18.1.6 in~\cite{85k:14004} for the 
key facts about the graph construction.) As in \S\ref{setup}, we denote by $Z$ the 
center of the blow-up, $E$ the exceptional divisors; and the natural morphisms as 
in this diagram:
\[
\xymatrix{
E \ar[d]_p \ar[r]^j & \hat V \ar[d]^\pi \\
Z \ar[r]_\iota & V
}
\]
By assumption $Z$ is a nonsingular irreducible subvariety of $V$; we let $r$ be its 
codimension. In a neighborhood of $Z$, $Z$ is the transversal intersection of $r$
components of $D$. The key lemma will be the following:

\begin{lemma}\label{keylemma}
Under the hypotheses of Claim~\ref{mainclaim}:
\begin{itemize}
\item There is a homomorphism of vector bundles $\sigma: \pi^*\Der_V(D)\to 
\Der_{\hat V}(D')$ that is an isomorphism in the complement of~$E$.
\item The restriction of $\sigma$ to $E$ induces 
a morphism of complexes of vector bundles
\[
\xymatrix{
\cO_E \ar[d]_{=} \ar@{^(->}[r] & \pi^*\Der_V(D)|_E \ar[d]^{\sigma|_E} \ar@{->>}[r] &
p^*\Der_Z \ar[d]_{=} \\
\cO_E \ar@{^(->}[r] & \Der_{\hat V}(\hat D)|_E \ar@{->>}[r] & p^*\Der_Z
}
\]
\end{itemize}
\end{lemma}

\noindent The monomorphisms and epimorphisms shown in this diagram
will be defined in the course of the proof of Lemma~\ref{keylemma}, in~\S\ref{proof2};
the monomorphisms will be monomorphisms of vector bundles.

Claim~\ref{mainclaim} follows from Lemma~\ref{keylemma}, as we now show.
Applying the graph construction to $\sigma$ yields a cycle $\sum_i a_i [W_i]$ 
of dimension $n=\dim \hat V$ in the Grassmannian 
$G=\text{Grass}_n(\pi^*\Der_V(D)\oplus \Der_{\hat V}(D'))$ over $\hat V$.
Since $\sigma$ is an isomorphism off $E$, 
\[
c(\pi^*\Der_V(D))\cap [\hat V] - c(\Der_{\hat V}(D'))\cap [\hat V]
=\sum_{W_i \to E} a_i \eta_{i*}(c(\zeta)\cap [W_i])\quad,
\]
where $\eta_i: W_i\to \hat V$ are the maps induced by projection, and
$\zeta$ is the rank-$n$ universal bundle on $G$. (See (c) in Example~18.1.6 
of~\cite{85k:14004}.) Pushing forward to $V$, and since $\Der_V(D)$ is assumed
to be locally free,
\[
c(\Der_V(D))\cap [V] - \pi_* (c(\Der_{\hat V}(D'))\cap [\hat V])
=\sum_{W_i \to E} a_i \pi_*\eta_{i*}(c(\zeta)\cap [W_i])\quad.
\]
Therefore, in order to verify Claim~\ref{mainclaim} it suffices to prove that
$\pi_*\eta_*(c(\zeta)\cap [W])=0$
for every component $W=W_i$ projecting into $E$ via $\eta=\eta_i$. We let
$\underline \eta$ be the morphism $W\to E$:
\[
\xymatrix{
W \ar[d]_{\underline\eta} \ar[dr]^\eta \\
E \ar[r]_j \ar[d]_p & \hat V \ar[d]^\pi \\
Z \ar[r]_\iota & V
}
\]
We have $\pi\circ\eta=\iota\circ p\circ \underline \eta$. Thus it suffices to show that
\[
p_*{\underline\eta}_*(c(\zeta)\cap [W])=0\quad.
\]
The component $W$ lies in the restriction 
$G|_E=\text{Grass}_n(\pi^*\Der_V(D)|_E\oplus \Der_{\hat V}(D')|_E)$, and $\zeta$
restricts to the universal bundle $\zeta|_E$ on $G|_E$; $c(\zeta)\cap [W]
=c(\zeta|_E)\cap [W]$ by functoriality of Chern classes. By the second part of 
Lemma~\ref{keylemma}, the restriction of $\zeta|_E$ to each component $W$ 
is the middle term in a complex of vector bundles
\[
\xymatrix{
\cO_E \ar@{^(->}[r] & \zeta|_E \ar@{->>}[r] & p^*\Der_Z\quad.
}
\]
It follows that $c(\zeta|_E)=c(p^*\Der_Z)\, c(\xi)$,
where $\xi=\ker(\zeta|_E \to p^*\Der_Z)/\cO_E$ is the homology of this complex. 
By the projection formula,
\[
p_*\underline\eta_*(c(\zeta)\cap [W])=c(\Der_Z)\cap p_*\underline\eta_*(c(\xi)\cap [W])
\quad.
\]
Since $\dim W=\dim V$ and $\xi$ has rank $=\codim_ZV-1$, the nonzero components
of $c(\xi)\cap [W]$ have dimension $>\dim Z$, and therefore
$p_*\underline\eta_*(c(\xi)\cap [W])=0$. It follows that 
$\pi_*\eta_*(c(\zeta)\cap [W])=0$ as needed.
\qed\medskip

\subsection{}
We are thus reduced to proving Lemma~\ref{keylemma}.
Recall that $Z$ denotes the codimension~$r$, nonsingular center of the blow-up. 
We will use the following notation:
\begin{itemize}
\item By assumption, there exist $r$ components $D_1,\dots,D_r$ of $D$ such that 
$Z$ is a connected component of $D_1\cap\cdots \cap D_r$. 
We will denote by $D^+$ the union of $D_1,\dots,D_r$. Note that $D^+$ is a 
divisor with normal crossings in a neighborhood of $Z$.
\item $\hat D$ will denote $\pi^{-1}(D)$, so that $D'=\hat D_{\text{red}}$.
\item Similarly, $\hat D^+$ will be $\pi^{-1}(D^+)$. 
\end{itemize}

\begin{remark}\label{redrem}
The difference between a divisor and its reduction is immaterial here
(in characteristic zero).
For a divisor $A$ in a nonsingular variety $V$, the sections of the sheaf $\Der_V(A)$ 
may be defined as those derivation which send a section $F$ corresponding to  
$A$ to a multiple of $F$: in other words, there is an exact sequence
\[
\xymatrix{
0 \ar[r] & \Der_V(A) \ar[r] & \Der_V \ar[r] & \cO_A(A)
}
\]
where (locally) the last map applies a given derivation to $F$
(see e.g.~\cite{MR2359100}, \S2). It is straightforward
to verify that if $\partial$ is a derivation, and $F_\text{red}$ consists of the factors
of $F$ taken with multiplicity~$1$, then $\partial(F)\in (F)$ if and only if
$\partial(F_\text{red})\in (F_\text{red})$. 
Thus $\Der_V(A)$ and $\Der_V(A_\text{red})$ coincide as subsheaves of
$\Der_V$. Therefore, we may use $\hat D$ in place of $D'$, and we don't
need to bother making a distinction between $\hat D^+$ and its reduction.
\qede\end{remark}

\begin{remark}
We recall the following useful description of $\Der_V(D)$ (cf.~\cite{MR1217488},
Proposition~4.8): if $D$ is the union of distinct components $D_i$, then
$\Der_V(D)=\cap_i \Der_V(D_i)$ within $\Der_V$. Indeed, it suffices
to prove that $\Der_V(A\cup B)=\Der_V(A)\cap \Der_V(B)$ if $A$ and $B$
have no components in common. This amounts to the statement that if $F$ 
and $G$ have no common factors, then $\partial(FG)\in (FG)$ if and only if 
$\partial(F)\in (F)$ and $\partial(G)\in (G)$ for all derivations $\partial$, which
is immediate.
\qede\end{remark}

\begin{remark}\label{inclus}
In particular, if a divisor $A$ consists of a selection of the components of~$D$, 
then $\Der_V(D)\subseteq \Der_V(A)$.
Therefore, we have inclusions
\[
\Der_V(D)\subseteq \Der_V(D^+)\quad,\quad
\Der_{\hat V}(\hat D)\subseteq \Der_{\hat V}(\hat D^+)
\quad.
\]
Further, the monomorphism $\Der_V(D)\hookrightarrow \Der_V(D^+)$
of locally free sheaves
remains a monomorphism after pull-back via $\pi$: the determinant of this
morphism is nonzero on $V$, and it remains nonzero on the blow-up $\hat V$.
\qede\end{remark}

\begin{lemma}\label{isoNC}
The (reduction of the) divisor $\hat D^+$ is a divisor with normal crossings in a 
neighborhood of $E$, and $\pi^* \Der_V(D^+)\cong \Der_{\hat V}(\hat D^+)$.
\end{lemma}

\begin{proof}
The first assertion is a simple verification in local coordinates. The second
assertion only need be verified in a neighborhood of $E$, so it reduces to 
the case of normal crossings, where it is straightforward. More details may
be found in Theorem~4.1 of~\cite{MR2504753}. (Also cf.~Lemma~1.3
in~\cite{MR2600139}.)
\end{proof}

We may use the isomorphism obtained in Lemma~\ref{isoNC} to identify
$\pi^* \Der_V(D^+)$ and $\Der_{\hat V}(\hat D^+)$. Via this identification,
we will verify that $\pi^*\Der_V(D)$ is contained into $\Der_{\hat V}(\hat D)$.
The corresponding monomorphism of locally free sheaves $\pi^*\Der_V(D)
\hookrightarrow \Der_{\hat V}(\hat D)=\Der_{\hat V}(D')$ will give the 
homomorphism $\sigma$ whose existence is claimed in Lemma~\ref{keylemma}.

Note that the sought-for $\sigma$ appears to go in the wrong direction. 
The differential of $\pi$
maps $\Der_{\hat V}$ to $\pi^* \Der_V$, and restricts to a homomorphism
$\Der_{\hat V}(D^+)\rightarrow \pi^* \Der_V(D^+)$. This is an isomorphism
as observed in Lemma~\ref{isoNC}, and the claim here is that its {\em inverse\/} 
restricts to a morphism
\[
\xymatrix{
\sigma: \pi^*\Der_V(D)\ar[r] & \Der_{\hat V}(\hat D)\quad,
}
\]
which will then clearly be an isomorphism off $E$ as needed in~\S\ref{graphsec}.

\begin{lemma}\label{sandwich}
In $\pi^* \Der_V(D^+)\cong \Der_{\hat V}(\hat D^+)$, $\pi^*\Der_V(D)
\subseteq \Der_{\hat V}(\hat D)$.
\end{lemma}

\begin{proof}
By definition of $\Der_V(D)$ there is an exact sequence
\begin{equation}
\tag{$\dagger$}
\xymatrix{
\Der_V(D) \ar[r] & \Der_V(D^+) \ar[r] & \cO_D(D)
}
\end{equation}
where the first map is a monomorphism, and the second applies a given logarithmic
derivation to a section~$F$ defining $D$. Pulling back to $\hat V$ gives a complex
\begin{equation}
\tag{$\ddagger$}
\xymatrix{
\pi^*\Der_V(D) \ar[r] & \pi^*\Der_V(D^+)\cong \Der_{\hat V}(\hat D^+) 
\ar[r] & \pi^*\cO_D(D)\cong \cO_{\hat D}(\hat D)
}
\end{equation}
The first map remains a monomorphism (Remark~\ref{inclus}), and maps
$\pi^*\Der_V(D)$ into the kernel of the second map, which is $\Der_{\hat V}(\hat D)$
by definition of the latter.
\end{proof}

This completes the proof of the first part of Lemma~\ref{keylemma}.
Note that $\sigma$ is a monomorphism of sheaves, not of vector bundles.

\begin{example}\label{threelinesex}
Let $V=\Pbb^2$, and let $D$ be the divisor consisting of three distinct concurrent
lines. We blow-up at the point of intersection $p$:
\begin{center}
\includegraphics[scale=.6]{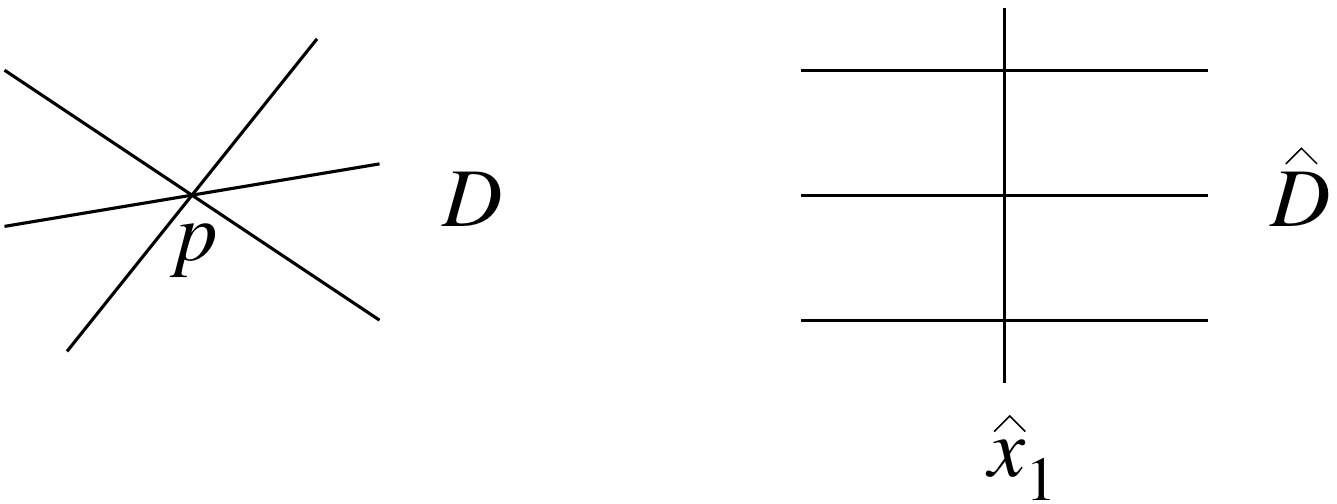}
\end{center}
In affine coordinates centered at $p$, we may assume $D$
has equation $F=x_1x_2 (x_1+x_2)=0$. We choose coordinates 
$\hat x_1, \hat x_2$ in an affine chart in the blow-up $\hat V$ so that the 
blow-up map is given by
\[
x_1=\hat x_1\quad,\quad x_2=\hat x_1\hat x_2\quad;
\]
the exceptional divisor $E$ has equation $\hat x_1=0$, and $\hat D$ is given by
the vanishing of $\hat F=\hat x_1^3 \hat x_2(1+\hat x_2)$ (the fourth component
is at $\infty$ in this chart); it is a divisor with normal crossings.

We work in the local rings $R$, $\hat R$ at $(0,0)$ in both $V$ and $\hat V$.
We can let $D^+$ be the divisor $x_1x_2=0$, so that $\hat D^+$ has ideal $(\hat x_1^2
\hat x_2)$. Bases for $\Der_V(D^+)$, $\Der_{\hat V}(\hat D^+)$ are
\[
\langle x_1\partial_1, x_2\partial_2\rangle\quad,\quad
\langle \hat x_1\hat \partial_1, \hat x_2\hat\partial_2\rangle
\]
where $\partial_i = \partial/\partial x_i$, $\hat \partial_i = \partial/\partial\hat x_i$,
and the isomorphism $\pi^*\Der_V(D^+) \overset\sim\to \Der_{\hat V}(\hat D^+)$
maps $\pi^*(x_1\partial_1)$ to $\hat x_1\hat\partial_1-\hat x_2\hat \partial_2$
and $\pi^*(x_2\partial_2)$ to $\hat x_2 \hat \partial_2$, as the reader may
verify. A derivation $a_1(x) x_1\partial_1+a_2(x) x_2\partial_2$ is in $\Der_V(D)$ iff
\[
(a_1(x) x_1\partial_1+a_2(x) \partial_2)(x_1x_2 (x_1+x_2))
=a_1(x) F+a_1(x) x_1^2 x_2 +a_2(x) F+a_2(x) x_1 x_2^2\in (F),
\]
that is, iff
\[
a_1(x) x_1+ a_2(x) x_2\in (x_1+x_2)\quad.
\]
It follows that a basis for $\Der_V(D)$ is
\[
\langle x_1\partial_1+ x_2\partial_2, (x_1+x_2)x_2\partial_2\rangle\quad,
\]
and we may represent sequence ($\dagger$) at $(0,0)$ as
\[
\xymatrix@C=75pt{
R\oplus R \ar[r]^-{\left(\begin{smallmatrix} 
1 & 0 \\
1 & x_1+x_2
\end{smallmatrix}\right)} & R\oplus R \ar[r]^-{\left(\begin{smallmatrix}
x_1^2 x_2 & x_1 x_2^2
\end{smallmatrix}\right)} & R/(F)
}
\]
Tensoring by $\hat R$ gives the corresponding sequence ($\ddagger$):
\[
\xymatrix@C=75pt{
\hat R\oplus \hat R \ar[r]^-{\left(\begin{smallmatrix} 
1 & 0 \\
1 & \hat x_1 (1+\hat x_2)
\end{smallmatrix}\right)} & \hat R\oplus \hat R \ar[r]^-{\left(\begin{smallmatrix}
\hat x_1^3 \hat x_2 & \hat x_1^3 \hat x_2^2
\end{smallmatrix}\right)} & \hat R/(\hat F)
}
\]
which realizes $\pi^* \Der_V(D)$ as a submodule of
\[
\Der_{\hat V}(\hat D)
=\ker(\begin{pmatrix} 
\hat x_1^3 \hat x_2 & \hat x_1^3 \hat x_2^2
\end{pmatrix})
=\im\left(
\begin{pmatrix} 
1 & 0 \\
1 & 1+\hat x_2
\end{pmatrix}
\right)
\]
In these coordinates, a matrix representation for 
$\pi^*\Der_V(D) \hookrightarrow \Der_{\hat V}(\hat D)$
is evidently $\left(\begin{smallmatrix}
1 & 0 \\
0 & \hat x_1
\end{smallmatrix}\right)$.
\qede
\end{example}

\begin{remark}
Example~\ref{threelinesex} illustrates the general local situation: dualizing 
Proposition~4.5 in \cite{MR1484696}, one sees that one may always choose local
coordinates and bases for $\Der_V(D)$, $\Der_{\hat V}(\hat D)$ so that the 
matrix of $\pi^*\Der_V(D) \hookrightarrow \Der_{\hat V}(\hat D)$ is diagonal, with
entries given by powers of the equation for the exceptional divisor.
This will not be needed in the following, but it is a useful model to keep in mind
in reading what follows.
\qede
\end{remark}

\subsection{}\label{proof2}
We are left with the task of proving the second part of Lemma~\ref{keylemma},
which amounts to the existence of a certain trivial subbundle and an epimorphism 
to $p^*\Der_Z$ for both $\pi^*\Der_V(D)$ and $\Der_{\hat V}(\hat D)$.
We will prove that there is a commutative diagram of locally free sheaves on $E$:
\[
\xymatrix{
\cO_E \ar@{^(->}[rr] \ar[dr] & & \Der_{\hat V}(\hat D)|_E \ar[dr] \\
&  \pi^*\Der_V(D)|_E \ar@{->>}[rr] \ar[ur]_{\sigma|_E} & & p^* \Der_Z
}
\]
such that the composition $\cO_E \to p^*\Der_Z$ is the zero morphism.
The top horizontal morphism will be a monomorphism of vector bundles,
and it follows from the commutativity of the diagram that so is the leftmost
slanted morphism. Similarly, the bottom horizontal morphism will be an
epimorphism, and it follows that so is the rightmost slanted morphism.
Thus, the full statement of Lemma~\ref{keylemma} follows from the 
existence of this diagram.

\subsubsection{} We deal with the epimorphism side first.
According to our hypotheses, the center $Z$ of the blow-up is the transversal
intersection (in a neighborhood of~$Z$) of the $r$ components of $D^+$, and is
contained in the other components of $D$. 
As $Z\subseteq V$, we have a natural embedding of $Der_Z\cong TZ$ as the
kernel of the natural map from $\Der_V|_Z\cong TV|_Z$ to the normal bundle
$N_ZV$.

\begin{lemma}\label{epinc}
There is an exact sequence of vector bundles
\[
\xymatrix{
0 \ar[r] & \cO_Z^{\oplus r} \ar[r] & \Der_V(D^+)|_Z \ar[r] & \Der_V|_Z
\ar[r] & N_ZV \ar[r] & 0 \quad.
}
\]
In particular, there is an epimorphism $\Der_V(D^+)|_Z \twoheadrightarrow
\Der_Z$.
\end{lemma}

\begin{proof}
We have (see Remark~\ref{redrem}) an exact sequence
\[
\xymatrix{
0 \ar[r] & \Der_V(D^+) \ar[r] & \Der_V \ar[r] & \cO_{D^+}(D^+)
}
\]
The image of the rightmost map is the ideal of $\cO_{D^+}(D^+)$ defined
locally by the partials of a generator for the ideal of $D^+$. Near $Z$, where
$Z$ is the complete intersection of $D_1,\dots,D_r$, it is
easy to verify that this ideal is isomorphic to $\oplus_{i=1}^r \cO_{D_i}(D_i)$.
Thus, tensoring by $\cO_Z$ gives an exact sequence
\[
\xymatrix@C=18pt{
0 \ar[r] & \Tor_1(\cO_Z,\oplus_{i=1}^r \cO_{D_i}(D_i)) \ar[r] 
& \Der_V(D^+)|_Z \ar[r] & \Der_V|_Z \ar[r] & \oplus_{i=1}^r \cO_{D_i}(D_i)|_Z
\ar[r] & 0
}
\]
(The leftmost term is $0$ as $\Der_V$ is locally free.) The term 
$\oplus_{i=1}^r \cO_{D_i}(D_i)|_Z$ is $N_ZV$, and the map from $\Der_V|_Z$
is the standard projection $TV|_Z \to N_ZV$. The $\Tor$ on the left is the
direct sum of $\Tor_1(\cO_Z,\cO_{D_i}(D_i))$, and it is easy to verify that
each such term is $\cong \cO_Z$, as claimed.
\end{proof}

\begin{remark}\label{coordepi}
We can choose local parameters $x_1,\dots, x_n$ for $V$ at a point of $Z$ such 
that $x_i$ is a generator for the ideal of $D_i$ for $i=1,\dots,r$. Then $\Der_V(D^+)$
has a basis given by derivations
\[
x_1 \partial_1\,,\,\dots, x_r\partial_r\,,\,\partial_{r+1}\,,\,\dots\,,\, \partial_n
\]
where $\partial_i=\partial/\partial x_i$. With the same coordinates, $\partial_{r+1},
\dots \partial_n$ restrict to a basis for $\Der_Z$, and the epimorphism
found in Lemma~\ref{epinc} acts in the evident way. The kernel is spanned
by the restrictions of $x_i \partial_i$, $i=1,\dots,r$; these are the $r$ trivial
factors appearing on the left in the sequence in Lemma~\ref{epinc}.

Also note that the `Euler derivation' $x_1\partial_1+\dots +x_r\partial_r$
spans a trivial subbundle $\cO_Z\hookrightarrow \cO_Z^{\oplus r}$ of
the kernel. Thus, we have a complex of vector bundles
\[
\xymatrix{
\cO_Z \ar@{^(->}[r] & \Der_V(D^+)|Z \ar@{->>}[r] & \Der_Z
}
\]
on $Z$. Pulling back to $E$, this gives a complex of vector bundles
\[
\xymatrix{
\cO_E \ar@{^(->}[r] & \pi^*\Der_V(D^+)|_E \ar@{->>}[r] & \pi^*\Der_Z
}
\]
on $E$. We have to verify that the same occurs for $\pi^*\Der_V(D)$
and $\Der_{\hat V}(\hat D)$.
\qede\end{remark}

Consider $\Der_V(D)$. We have (Remark~\ref{inclus}) inclusions
\[
\Der_V(D) \subseteq \Der_V(D^+) \subseteq \Der_V\quad.
\]
Restricting to $Z$, and in view of Lemma~\ref{epinc}, we get morphisms
\[
\xymatrix{
\Der_V(D)|_Z \ar[r] & \Der_V(D^+)|_Z \ar@{->>}[r] & \Der_Z\quad.
}
\]

\begin{claim}
The composition $\Der_V(D)|_Z \to \Der_Z$ is an epimorphism.
\end{claim}

\begin{proof}
Working with local parameters as in Remark~\ref{coordepi}, it suffices to
note that the derivations $\partial_{r+1},\dots,\partial_n$ are in $\Der_V(D)$:
this is clear, since by assumption $D$ admits a local generator of the form
$x_1\cdots x_r\, G(x_1,\dots,x_r)$.
\end{proof}

Pulling back to $E$ and using Lemma~\ref{sandwich} we get morphisms
\[
\xymatrix@H=30pt{
\pi^* \Der_V(D)|_E \ar[r]^{\sigma|_E} \ar@/_1.5pc/@{->>}[rrr] & 
\Der_{\hat V}(\hat D)|_E \ar[r] & \pi^* \Der_V(D^+) \ar@{->>}[r] & p^*\Der_Z
}
\]
and this yields the commutative triangle on the right in the diagram at the
beginning of the section.

\subsubsection{} Finally, we have to deal with the triangle on the left.

\begin{lemma}\label{onehyp}
Let $A$ be a nonsingular hypersurface of a nonsingular variety $V$. Then
there is an exact sequence of vector bundles
\[
\xymatrix{
0 \ar[r] & \cO_A  \ar[r] & \Der_V(A)|_A \ar[r] & 
\Der_A  \ar[r] & 0\quad.
}
\]
\end{lemma}

\begin{proof}
In the case of a nonsingular hypersurface, the basic sequence of 
Remark~\ref{redrem} is also exact on the right:
\[
\xymatrix{
0 \ar[r] & \Der_V(A) \ar[r] & \Der_V \ar[r] & \cO_A(A) \ar[r] & 0
}
\]
Restricting to $A$ gives the exact sequence
\[
\xymatrix{
0 \ar[r] & \Tor_1^{\cO_V}(\cO_A(A),\cO_A) \ar[r] & \Der_V(A)|_A \ar[r] & 
\Der_V|_A \ar[r] & \cO_A(A) \ar[r] & 0\quad.
}
\]
The kernel of the rightmost nonzero map is $\Der_A$ (note that $\cO_A(A)$ 
is the normal bundle of $A$ in $V$), and an immediate computation gives
$\Tor_1^{\cO_V}(\cO_A(A),\cO_A)\cong \cO_A$, so we get the exact
sequence in the statement. This is an exact sequence of vector bundles
since the kernel of a surjective map of vector bundles is a subbundle.
\end{proof}

\begin{remark}\label{coordmono}
Applying this lemma to $E\subseteq \hat V$ gives a distinguished copy of
$\cO_E$ in $\Der_{\hat V}(E)|_E$. Adopting local parameters at a point
of $Z$ as in Remark~\ref{coordepi}, we can choose coordinates 
\[
\hat x_1\,,\, \hat x_2\,,\, \dots \,,\, \hat x_r \,,\, \hat x_{r+1} \,,\, \dots \,,\, \hat x_n
\]
at a point of $E$ in a chart of the blow-up $\hat V$ so that the blow-up map 
is given by
\[
\left\{
\aligned
x_1 &=\hat x_1 \\
x_i &=\hat x_1 \hat x_i\quad\, i=2,\dots,r \\
x_j &=\hat x_j \qquad j=r+1,\dots,n
\endaligned \right.
\]
The exceptional divisor is given by $\hat x_1=0$. Then a basis for 
$\Der_{\hat V}(E)$ at this point is
\[
\hat x_1 \hat \partial_1\,,\, \hat \partial_2\,,\, \dots\,,\, \hat \partial_n
\]
where $\hat \partial_i = \partial/\partial\hat x_i$. The distinguished copy of
$\cO_E$ in $\Der_{\hat V}(E)$ found in Lemma~\ref{onehyp} is spanned 
by $\hat x_1\hat \partial_1$.
\qede\end{remark}

Now recall (Remark~\ref{inclus}) that $\Der_{\hat V}(\hat D)\subseteq \Der_{\hat V}(E)$.

\begin{claim}
The distinguished $\cO_E\subseteq \Der_{\hat V}(E)|_E$ is contained in 
$\Der_{\hat V}(\hat D)|_E$.
\end{claim}

\begin{proof}
We work in coordinates as in Remark~\ref{coordepi} and~\ref{coordmono}. 
By hypothesis, $D$ is given analytically by the vanishing of 
$F=x_1\cdots x_r\cdot G(x_1\cdots x_r)$, where $G$ is homogeneous.
In the chart considered above in the blow-up, $\hat D$ is therefore given by 
the vanishing of
\[
\hat F=\hat x_1^m \hat x_2 \cdots \hat x_r\, G(1,\hat x_2,\dots, \hat x_r)
\]
where $m$ is the multiplicity of $D$ along $Z$. We then see that
\[
(\hat x_1\hat \partial_1) \hat F=m \,\hat x_1^m \hat x_2 \cdots \hat x_r 
\, G(1,\hat x_2,\dots, \hat x_r)=m \hat F\in (\hat F)\quad:
\]
this shows that $\hat x_1\hat \partial_1\in \Der_{\hat V}(\hat D)$, as claimed.
\end{proof}

Since $\Der_{\hat V}(\hat D)\subseteq \Der_{\hat V}(\hat D^+)$, and the
latter is isomorphic to $\pi^*\Der_V(D^+)$, we can view the distinguished
$\cO_E$ as a subsheaf of $\pi^*\Der_V(D^+)|_E$. Chasing coordinates,
it is straightforward to check that
\[
\hat x_1\hat\partial_1 \mapsto x_1\partial_1 + \cdots + x_r\partial_r\quad:
\]
that is, this copy of $\cO_E$ corresponds to the `Euler derivation' 
identified in Remark~\ref{coordepi}. Further, we see that it is also contained
in $\pi^*\Der_V(D)|_E$: indeed, since in the chosen analytic coordinates $F$ is
homogeneous (up to factors not vanishing along $Z$), the Euler derivation 
acts on $F$ by multipliying it by its degree, and hence it lands in the ideal $(F)$. 

At this point we have the following situation:
\[
\xymatrix@H=30pt{
\cO_E \ar[r] \ar@/^1.5pc/@{^(->}[rrr] 
& \pi^*\Der_V(D)|_E \ar[r]^{\sigma|_E} &
\Der_{\hat V}(\hat D)|_E \ar[r] &
\pi^*\Der_V(\hat D^+)|_E
}
\]
This yields the commutative triangle on the left in the diagram at the
beginning of the section. 
(The above construction shows that the monomorphisms from $\cO_E$ to 
$\Der_{\hat V}(\hat D)|_E$ and $\pi^*\Der_V(\hat D^+)|_E$ are monomorphisms 
of vector bundles, as needed.)
The composition with the projection to $p^*\Der_Z$
is $0$ as noted in Remark~\ref{coordepi}, so this completes the proof of
Lemma~\ref{keylemma}. Claim~\ref{mainclaim} follows from 
Lemma~\ref{keylemma} as shown in~\S\ref{graphsec}, so this concludes
the proof of~Theorem~\ref{main}.


\section{Further remarks and examples}\label{remandex}

\subsection{An example}
We illustrate Theorem~\ref{main} by computing
the Chern class of a sheaf of logarithmic derivations in a simple case.
Any value Theorem~\ref{main} may have lies in the contrast between the
standard computation, by means of the basic sequence defining the
sheaf, and the computation using Chern-Schwartz-MacPherson classes,
which has a very different, `combinatorial' flavor.

We assume $D$ consists of $m\ge 2$ nonsingular components $D_i$, each
of class $X$, meeting pairwise transversally along a codimension-$2$ nonsingular
subvariety $Z$.\medskip

{\sc ---Computation using $\csm$-classes.\/}
As $D=\cup_i D_i$, and since all components meet along $Z$, we have
\[
\one_D=\one_Z+\sum_i \one_{D_i\smallsetminus Z}=(\sum_i \one_{D_i})
-(m-1) \one_Z\quad,
\]
and hence
\[
\one_{V\smallsetminus D}=\one_V-\sum_i \one_{D_i} + (m-1)\one_Z\quad.
\]
Since $V$, all $D_i$, and $Z$ are nonsingular, the basic normalization property
of $\csm$ classes (\S\ref{csmsum}) gives
\[
\csm(\one_{V\smallsetminus D})
=c(TV)\cap [V]-\sum_i c(TD_i)\cap [D_i] + (m-1) c(TZ)\cap [Z]\quad.
\]
We are assuming that all components have the same class $X$, and hence
$Z$ has class $X\cdot X$. Thus, this gives
\[
\csm(\one_{V\smallsetminus D})
=c(TV)
\left(1-\sum_{i=1}^m \frac{X}{1+X}+(m-1) \frac{X^2}{(1+X)^2}
\right)\cap [V]
\]
According to Theorem~\ref{main}, this class equals $c(\Der_V(D))\cap [V]$.
That is,
\[
c(\Der_V(D))=\frac{c(TV)\, (1-(m-2) X)}{(1+X)^2} \quad.
\]

{\sc Standard computation.}
The basic sequence recalled in Remark~\ref{redrem} may be completed to
\[
\xymatrix{
0 \ar[r] & \Der_V(D) \ar[r] & \Der_V \ar[r] & \cO_D(D) \ar[r] & \cO_{JD}(D)
\ar[r] & 0\quad,
}
\]
where $JD$ is the {\em singularity subscheme\/} (or {\em Jacobian subscheme})
of $D$. Therefore,
\[
c(\Der_V(D))=\frac{c(\Der_V)}{c(\cO_D(D))}\, c(\cO_{JD}(D))
=\frac{c(TV)}{1+D}\, c(\cO_{JD}(D))\quad.
\]
In the case at hand, we are assuming that $D$ is defined by a section
$f_1\cdots f_m$ of $\cO(mX)$; $Z$ is defined by (say) $f_1=f_2=0$,
meeting transversally at every point of~$Z$;
and $f_i=(a_i f_1+b_i f_2)$ for $i\ge 3$, without multiple components. 
It follows that $JD$ is a complete intersection of two sections of
$\cO((m-1)X)$, so $\cO_{JD}$ is resolved by a Koszul complex:
\[
\xymatrix@C=13pt{
0 \ar[r] & \cO_V(-2(m-1)X) \ar[r] & \cO_V(-(m-1)X)\oplus \cO_V(-(m-1)X) \ar[r] &
\cO_V \ar[r] & \cO_{JD} \ar[r] & 0
}
\]
and twisting by $\cO_V(D)=\cO_V(mX)$ gives the exact sequence
\[
\xymatrix@C=20pt{
0 \ar[r] & \cO_V((-m+2)X) \ar[r] & \cO_V(X)\oplus \cO_V(X) \ar[r] &
\cO_V(D) \ar[r] & \cO_{JD}(D) \ar[r] & 0\quad.
}
\]
Thus
\[
c(\cO_{JD}(D))=\frac{c(\cO_V(D))\, c(\cO(-(m-2))X)}{c(\cO_V(X))^2}
=\frac{(1+D)(1-(m-2)X)}{(1+X)^2}\quad.
\]
Taking this into account in the expression for $c(\Der_V(D))$ given above,
we recover the result of the $\csm$ computation.
\qede\medskip

While this may be largely a matter of taste, the standard computation appears
to us to involve subtler information than the alternative combinatorial 
computation via $\csm$ classes afforded by applying Theorem~\ref{main}. 
The point is that the $\csm$ class already includes information on the singularity 
subscheme $JD$: see \cite{MR2001i:14009} for the precise relation.
Computing the $\csm$ class, which is straightforward for a hypersurface 
arrangement, takes automatically care of accounting for the total Chern class
of~$\cO_{JD}(D)$.

\subsection{A projection formula}
If $\cE$ is a vector bundle on a scheme $X$, and $\alpha: Y \to X$ is a proper morphism,
then for any class $A$ in the Chow group of $Y$ we have
\[
\alpha_*(c(\alpha^*\cE)\cap A)=c(\cE)\cap \alpha_* (A)\quad.
\]
This is a basic result on Chern classes, see~Theorem~3.2 (c) in~\cite{85k:14004}.
On a nonsingular variety, a notion of total Chern class is available for all coherent 
sheaves: this follows from the isomorphism $K_0(V)\cong K^0(V)$ for $V$ nonsingular
(\cite{85k:14004}, \S15.1) and the Whitney formula. However, a straightforward 
projection formula as in the case of vector bundles does not hold for arbitrary 
coherent sheaves.

\begin{example}
Let $V$ be nonsingular, let $X,Y\hookrightarrow V$ be distinct irreducible hypersurfaces,
and let $i: X\hookrightarrow V$ be the inclusion. From the exact sequence
\[
\xymatrix{
0 \ar[r] & \cO_V(-X) \ar[r] & \cO_V \ar[r] & \cO_X \ar[r] & 0
}
\]
it follows that $c(\cO_X)=\frac 1{1-X}$, and similarly $c(\cO_Y)=\frac 1{1-Y}$.
As $i^*(\cO_X)=\cO_X$, we see that
\[
i_*(c(i^*\cO_X)\cap [X])=i_*([X])\ne c(\cO_X)\cap i_*([X])=\frac{[X]}{1-X}\quad:
\]
the projection formula does not hold in this case. On the other hand,
$i^*(\cO_Y)=\cO_{X\cap Y}$, and $X\cap Y$ is a divisor in $X$ with bundle
$\cO_X(X\cap Y)=i^* \cO_V(Y)$; therefore,
\[
i_*(c(i^*\cO_Y)\cap [X])=i_*\left(\frac{[X]}{1-i^*Y}\right)=\frac{[X]}{1-Y}
= c(\cO_Y)\cap i_*([X])\quad:
\]
the projection formula does hold in this case.
\qede\end{example}

The difference between the two cases considered in this example is a matter of
$\Tor$ functors: $\Tor_1^{\cO_V}(\cO_X,\cO_X)\cong \cO_X(-X)$ is not trivial, while
$\Tor_1^{\cO_V}(\cO_X,\cO_Y)$ vanishes. It is essentially evident from the definitions
that for a coherent sheaf $\cF$ on $V$, and a morphism $\alpha: W\to V$,
\[
\alpha^* c(\cF)=\prod_{i\ge 0} c(\Tor_i^{\cO_V}(\cO_W,\cF))^{(-1)^i}
=c(\alpha^*\cF)\cdot \prod_{i\ge 1} c(\Tor_i^{\cO_V}(\cO_W,\cF))^{(-1)^i}\quad,
\]
and in particular $c(\alpha^*\cF)=\alpha^* c(\cF)$ if the higher $\Tor$s vanish.
In this case (for example, in the case of vector bundles) the projection formula holds
if $\alpha$ is proper. More generally, the projection formula holds if $\alpha_*$
maps to $1$ the total Chern classes of the higher $\Tor$s.

\subsection{}
Now recall the situation of this paper, and particularly the blow-up considered in
Claim~\ref{mainclaim} and \S\ref{proof}: $D$ is a hypersurface arrangement
in a nonsingular variety $V$, and $Z$ is an intersection of minimal dimension of
components of $D$. In fact, $Z=D_1\cap \cdots \cap D_r$, where $D_1,\dots,
D_r$ are components of $D$ meeting with normal crossings in a neighborhood
of $Z$. We denote by $D^+$ the union of these components, and we have 
observed (Remark~\ref{inclus}) that $\Der_V(D)\subseteq \Der_V(D^+)$.

The sections obtained by applying the derivations in $\Der_V(D^+)$ to a 
section $F$ defining $D$, together with $F$, define a subscheme $J^+D$ 
of $\cO_D$, which should
be viewed as a `modified Jacobian subscheme' of the hypersurface 
arrangement $D$ (depending on the choice of the subdivisor $D^+$). 
We consider the coherent sheaf $\cO_{J^+D}(D)$ in $V$.

Finally, recall that $\pi: \hat V\to V$ denotes the blow-up of $V$ along $Z$.

\begin{claim}\label{profor}
Theorem~\ref{main} is equivalent to the statement that, for all blow-ups as
above, $\cO_{J^+D}(D)$ satisfies the projection formula with respect
to the blow-up map $\pi$:
\[
\pi_*(c(\pi^* \cO_{J^+D}(D))\cap [\hat V]) = c(\cO_{J^+D}(D))\cap [V]\quad.
\]
\end{claim}

\begin{proof}
Arguing as in \S\ref{setup}, we only need to deal with the case of a single 
blow-up, and show that the given formula is equivalent to the formula in 
Claim~\ref{mainclaim}. 

Restricting the basic sequence recalled in Remark~\ref{redrem} to $D^+$ 
gives an exact sequence
\[
\xymatrix{
0 \ar[r] & \Der_V(D) \ar[r] & \Der_V(D^+) \ar[r] & \cO_D(D)
}
\]
and the image of the last morphism is the ideal generated by applying
the derivations from $\Der_V(D^+)$ to $F$; this ideal defines $J^+D$, so
we have an exact sequence
\[
\xymatrix{
0 \ar[r] & \Der_V(D) \ar[r] & \Der_V(D^+) \ar[r] & \cO_D(D) \ar[r] &
\cO_{J^+D}(D) \ar[r] & 0
}
\]
on $V$. Notice that this implies that
\[
c(\Der_V(D))=\frac{c(\Der_V(D^+))}{1+D}\, c(\cO_{J^+D}(D))\quad.
\]
Now we claim that (with notation as in \S\ref{proof}) there is an exact 
sequence
\begin{equation*}
\tag{*}
\xymatrix@C=20pt{
0 \ar[r] & \Der_{\hat V}(\hat D) \ar[r] & \pi^*\Der_V(D^+) \ar[r] & \
\pi^*\cO_D(D) \ar[r] & \pi^*\cO_{J^+D}(D) \ar[r] & 0\quad.
}
\end{equation*}
Indeed, pulling back the last terms of the previous sequence to $\hat V$ 
gives the last terms of~(*), by right-exactness of $-\otimes_{\cO_V} \cO_{\hat V}$;
via the isomorphisms $\pi^*\Der_V(D^+)\cong \Der_{\hat V}(\hat D^+)$
(Lemma~\ref{isoNC}) and $\pi^*\cO_D(D)\cong \cO_{\hat D}(\hat D)$, 
the morphism in the middle is seen to act by applying derivations from 
$\Der_{\hat V}(\hat D)$ to a section defining $\hat D$. Hence its kernel is 
$\Der_{\hat V}(\hat D)$, as needed for~(*).
From (*), it follows that
\[
c(\Der_{\hat V}(\hat D))=\frac{c(\pi^*\Der_V(D^+))}{1+\pi^* D}\, 
c(\pi^* \cO_{J^+D}(D))\quad,
\]
and hence, applying the ordinary projection formula (as $\Der_V(D^+)$
is locally free)
\[
\pi_*( c(\Der_{\hat V}(\hat D))\cap [\hat V])=
\frac{c(\Der_V(D^+))}{1+D}\, \pi_*(c(\pi^*\cO_{J^+D}(D))\cap [\hat V])\quad.
\]
Comparing with the previous equality of Chern classes, we see that the
projection formula for $\cO_{J^+D}(D)$,
\[
\pi_*(c(\pi^*\cO_{J^+D}(D))\cap [\hat V])=c(\cO_{J^+D}(D))\cap [V]
\]
is equivalent to
\[
\pi_*( c(\Der_{\hat V}(\hat D))\cap [\hat V])= c(\Der_{\hat V}(\hat D))\cap [V]
\quad,
\]
that is the formula in Claim~\ref{mainclaim}, as claimed.
(As pointed out in Remark~\ref{redrem}, we can replace $\hat D$ for $D'$ in Claim~\ref{mainclaim}.) 
\end{proof}

By Claim~\ref{profor}, an independent proof of the projection formula for 
$\cO_{J^+D}(D)$ would give an alternative proof of Theorem~\ref{main}.
Note that the relevant $\Tor$ does not vanish in general; the task amounts
to showing that its Chern class pushes forward to $1$. We were not able to 
construct a more direct proof of this fact.

\begin{remark}
Xia Liao has shown (\cite{liao}) that the equality in Theorem~\ref{main},
for any divisor $D$, is equivalent to a projection formula involving the blow-up 
along the (ordinary) Jacobian subscheme of $D$.
\qede\end{remark}


\end{document}